\documentclass[10pt,a4paper, reqno]{amsart}
\usepackage{amsmath,amsthm,amssymb,mathrsfs,stmaryrd,color}
\usepackage[all]{xy}
\usepackage{url}
\usepackage{slashed}
\usepackage{geometry}

\usepackage[utf8]{inputenc}
\usepackage[T1]{fontenc}

\usepackage{relsize}
\usepackage[bbgreekl]{mathbbol}
\usepackage{amsfonts}

\usepackage{tikz-cd}
\usepackage{enumerate}

\DeclareMathOperator{\K}{K}
\DeclareMathOperator{\THH}{THH}

\DeclareMathOperator{\TR}{TR}
\DeclareMathOperator{\TP}{TP}
\DeclareMathOperator{\TC}{TC}
\DeclareMathOperator{\HH}{HH}

\DeclareMathOperator{\Calg}{CAlg}

\DeclareMathOperator{\spec}{Spec}

\DeclareMathOperator{\spf}{spf}

\DeclareMathOperator{\fib}{fib}

\DeclareMathOperator{\Fil}{Fil}
\DeclareMathOperator{\gr}{gr}

\DeclareMathOperator{\F}{F}
\DeclareMathOperator{\V}{V}

\DeclareMathOperator{\W}{W}

\newcommand{\FilM}{\Fil_{\mathcal{M}}}

\newcommand{\grM}{\gr_{\mathcal{M}}}

\newcommand{\Z}{\mathbb{Z}}
\newcommand{\Q}{\mathbb{Q}}

\newcommand{\Sph}{\mathbb{S}}

\newcommand{\Del}{\mathbb{\Delta}}
\newcommand{\Delc}{\widehat{\Del}}

\newcommand{\Zp}{\mathbb{Z}_p}

\theoremstyle{plain}
\newtheorem{thm}{Theorem}[section]
\newtheorem*{definition*}{Definition}

\newtheorem{prop}[thm]{Proposition}
\newtheorem{cor}[thm]{Corollary}

\theoremstyle{definition}
\newtheorem{defn}[thm]{Definition}

\newtheorem{con}[thm]{Construction}
\newtheorem{ex}[thm]{Example}

\DeclareMathOperator{\ainf}{\mathbb{A}_{\inf}}
\DeclareMathOperator{\einfty}{\mathbb{E}_{\infty}}

\usepackage{hyperref}

\usepackage[
backend=biber,
style=alphabetic,
sorting=nyt
]{biblatex}
\DeclareMathOperator{\Res}{R}

\usepackage{hyperref}
\DeclareMathOperator{\can}{can}

\usepackage{xcolor}

\DeclareMathOperator{\limr}{\underset{\Res}{\lim}}
\DeclareMathOperator{\rlimr}{R\underset{\Res}{\lim}}

\newcommand{\nr}{\mathcal{N}_r}
\newcommand{\ninf}{\mathcal{N}_{\infty}}
\newcommand{\n}{\mathcal{N}}
\newcommand{\HT}{\mathrm{HT}}
\newcommand{\conj}{\mathrm{conj}}

\newcommand{\drw}{\mathrm{dRW}}

\usepackage{marginnote}

\title[TR with logarithmic poles and the de Rham--Witt complex]{TR with logarithmic poles and the de Rham--Witt complex}
\author{Faidon Andriopoulos}
\thanks{\textit{Email address:} \texttt{fandri@uchicago.edu}}

\begin{document}

\maketitle

\begin{abstract}
   In the article of Hesselholt \cite{hesselholt2005absolute}, a set of conjectures is laid out. Given a smooth scheme $X$ over the ring of integers $\mathcal{O}_K$ of a $p$-adic field $K$, these conjectures concern the expected relation between log topological restriction homology $\TR^r (X,M_X)$ and the absolute log de Rham--Witt complex $W_r\Omega_{(X,M_X)}$.
   In this note, which is companion to \cite{Andri1}, we discuss the case of a $p$-completely smooth $p$-adic formal scheme $X$ over $\spf \mathcal{O}_C$, where $C$ is the field of $p$-adic complex numbers.
   Along the way, we study the motivic filtration of log $\TR^r$ and its $S^1$-homotopy fixed points, following ideas of \cite{binda2023logarithmic}.
\end{abstract}

\tableofcontents


\section{Introduction}
Let $K$ be a $p$-adic field with ring of integers $\mathcal{O}_K$ and residue field $k$. Given a smooth scheme $X$ over $\mathcal{O}_K$, one can associate to it its special fibre $Y$ and its generic fibre $U$, which naturally fit in the following diagram:
\begin{equation*}
    \begin{tikzcd}[column sep=huge]
        Y \arrow[r, hook, "i"] \arrow[d] & X \arrow[d] & U \arrow[l, hook', "j"'] \arrow[d] \\
        \spec k \arrow[r, hook] & \spec \mathcal{O}_K & \spec K \arrow[l, hook']
    \end{tikzcd}
\end{equation*}
Under this setup, in \cite{hesselholt2005absolute} Hesselholt lays out a set of conjectures, which concern the expected relation between $\TR^r$ with logarithmic poles and the absolute de Rham--Witt complex associated to such a diagram.

In the current note, as a brief companion to \cite{Andri1}, we discuss a version of these conjectures in the case when $X$ is a $p$-completely smooth $p$-adic formal scheme over $\mathcal{O}_C$, where $C$ is an algebraically closed complete extension of $\Q_p$. We work under the additional assumption that the algebraic approach \cite{vcesnavivcius2019cohomology, koshikawa2020logarithmic, koshikawa2023logarithmic} and the homotopy theoretic approach \cite{binda2023logarithmic} to completed log prismatic cohomology over a perfectoid base coincide (as expected). Moreover, for matters of simplicity, from now on we treat the affine case and assume that $X = \spf S$, where $S$ is the $p$-adic completion of a smooth $\mathcal{O_C}$-algebra.

\begin{thm} \label{thm1 Hesselholt conjectures}
    Let $X= \spf S$ be a $p$-completely smooth, affine $p$-adic formal scheme over $\spf \mathcal{O}_C$, as above. Let $Y$ be its special fibre and $U$ its generic fibre. Note that $X$ obtains a canonical log structure $(X,M_X)$, via $U$. The following statements hold \'etale locally for $X$:
    \begin{enumerate}
        \item $\TR^r ((X,M_X);\Zp)$ is equipped with a complete, exhaustive, descending, multiplicative motivic filtration, whose graded pieces are \'etale locally identified with the graded pieces of the $r$-Nygaard filtered log prismatic cohomology (which is identified with the $A\Omega$-cohomology):
        \begin{equation*}
            \grM^n \TR^r ((-,-);\Zp) \simeq \nr^n \Delc_{(-,-)/ \ainf} \simeq \tau^{\leq n} A\Omega_{(-,-)}/ \widetilde{\xi}_r
        \end{equation*}
        Thus, they also give rise to a spectral sequence, which degenerates on the second page:
        \begin{equation*}
            E_{s,t}^2 = H^{t-s} \big( \nr^n \Delc_{(-,-)/ \ainf} \big) \Rightarrow \TR_{s+t}^r ((-,-);\Zp)
        \end{equation*}
        \item There exists a comparison for these objects with the $p$-complete, relative, log de Rham--Witt complex of Matsuue \cite{matsuue2017relative}:
        \begin{equation*}
            H^n \big( \nr^n \Del_{(-,-)/ \ainf} \big) \simeq W_r\Omega_{(-,-)/\mathcal{O}_C}^{n, \mathrm{cont}}
        \end{equation*}
        which also sit on the following Frobenius equalizer sequence:
        \begin{equation*}
            \begin{tikzcd}
                \tau^{\leq n} i^* j_* \mu_{p^v}^{\otimes n} \arrow[r] & i^* \ninf^n \Delc_{(-,-)/ \ainf} / p^v \arrow[r, "1-\F"] & i^* \ninf^n \Delc_{(-,-)/ \ainf} / p^v
            \end{tikzcd}
        \end{equation*}
    \end{enumerate}
\end{thm}

In order to tackle this, we first need to define and study the log versions of the $r$-nygaard filtered prismatic cohomology and the motivic filtration of $\TR^r$, at least over a fixed perfectoid base. Building on the previous work \cite{Andri1, andriopoulos2024motivic}, this is contained in the following theorem:

\begin{thm} \label{thm2 motivic filtrations}
    Let $(S,Q)$ be a $p$-complete quasisyntomic pre-log ring, over a fixed perfectoid ring $R_0$, with associated perfect prism $(\ainf (R_0), \xi)$. The following hold:
    \begin{enumerate}
        \item For $1 \leq r \leq \infty$, the invariants $\TR^r ((S,Q);\Zp)^{hS^1} \to \TR^r ((S,Q);\Zp)$ are equipped with complete, exhaustive, descending, multiplicative, $\Z$-indexed motivic filtrations:
        \begin{equation*}
            \FilM^{\bullet} \TR^r ((S,Q);\Zp)^{hS^1} \longrightarrow \FilM^{\bullet} \TR^r ((S,Q);\Zp)
        \end{equation*}
        which come from log quasisyntomic sheafification of their respective double-speed Postnikov filtrations, from the log quasiregular-semiperfectoid case. Passing to the associated graded pieces, these can be identified with:
        \begin{equation*}
            \begin{cases}
                \grM^n \TR^r ((S,Q);\Zp)^{hS^1} \simeq R\Gamma \Big( (S,Q), \tau_{[2n-1,2n]} \TR^r ((-,-);\Zp)^{hS^1} \Big) \\[5pt]
                \grM^n \TR^r ((S,Q);\Zp) \simeq R\Gamma \Big( (S,Q), \tau_{[2n-1,2n]} \TR^r ((-,-);\Zp) \Big)
            \end{cases}
        \end{equation*}

        \item We denote by $\grM^{n, \mathrm{odd}}$ (resp. $\grM^{n, \mathrm{even}}$) the information that comes by applying log quasisyntomic sheafification to the odd (resp. even) homotopy groups of either $\TR^r ((-,-);\Zp)^{hS^1}$ or $\TR^r ((-,-);\Zp)$. Then $\grM^{n, \mathrm{odd}}$ corresponds to the $0^{\mathrm{th}}$ cohomology group of the two term complex and $\grM^{n, \mathrm{even}}$ to the $1^{\mathrm{st}}$ cohomology group.

        Let us focus on the case of $1 \leq r < \infty$ and the even parts of the associated graded pieces of the motivic filtration. Note that in the non-logarithmic case, the odd parts vanish locally in the quasisyntomic topology. There exists a filtration on completed log prismatic cohomology, which we call the \emph{log $r$-Nygaard filtration} and is given by the following iterated pullback involving $r$ terms:
        \begin{gather*}
            \nr^{\geq n} \Delc_{(S,Q)/ \ainf} \{ n\} := \\
            \n^{\geq n} \Delc_{(S,Q)/ \ainf} \{ n\} \times_{\Delc_{(S,Q)/ \ainf} \{ n\}} \dots \times_{\Delc_{(S,Q)/ \ainf} \{ n\}} \n^{\geq n} \Delc_{(S,Q)/ \ainf} \{ n\}
        \end{gather*}
        This is equipped with natural Restriction, Frobenius, and Verschiebung operators. Its filtered and graded pieces arise via the homotopy theoretic machinery in the following manner. We consider the commutative square:
        \begin{equation*}
            \begin{tikzcd}[column sep=huge]
                \TR^r ((S,Q);\Zp)^{hS^1} \arrow[r] \arrow[d] & \TC^- ((S,Q);\Zp) \arrow[r, "\varphi^{hS^1}"] \arrow[d] & \TP ((S,Q);\Zp) \arrow[d] \\
                \TR^r ((S,Q);\Zp) \arrow[r] & \THH ((S,Q);\Zp)^{hC_{p^{r-1}}} \arrow[r, "\varphi^{hC_{p^{r-1}}}"] & \THH ((S,Q);\Zp)^{tC_{p^r}}
            \end{tikzcd}
        \end{equation*}
        Passing to the even terms of the associated graded pieces of the motivic filtrations, we obtain the following commutative square involving log $r$-Nygaard filtered prismatic cohomology:
        \begin{equation*}
            \begin{tikzcd}[column sep=huge]
                \nr^{\geq n} \Delc_{(S,Q)/ \ainf} \{ n\} [2n] \arrow[r, "\varphi_{r, n}"] \arrow[d] & \Delc_{(S,Q)/ \ainf} \{ n\} [2n] \arrow[d] \\
                \nr^n \Delc_{(S,Q)/ \ainf} \{ n\} [2n] \arrow[r, "\gr \varphi_{r, n}"] & \Del_{(S,Q)/ \ainf}^{\HT, r} \{ n\} [2n]
            \end{tikzcd}
        \end{equation*}
        where on the upper row we have a graded version of the $r^{\mathrm{th}}$ iteration of Frobenius mapping from the $r$-Nygaard filtered prismatic cohomology, while on the lower row we have its graded counterpart.
        
        \item Taking the limit with respect to the restriction maps gives rise to the following filtration:
        \begin{gather*}
            \ninf^{\geq n} \Delc_{(S,Q)/ \ainf} \{ n\} := \rlimr \nr^{\geq n} \Delc_{(S,Q)/ \ainf} \{ n\} \simeq \\
            \dots \times_{\Delc_{(S,Q)/ \ainf} \{ n\}} \n^{\geq n} \Delc_{(S,Q)/ \ainf} \{ n\} \times_{\Delc_{(S,Q)/ \ainf} \{ n\}} \dots \times_{\Delc_{(S,Q)/ \ainf} \{ n\}} \n^{\geq n} \Delc_{(S,Q)/ \ainf} \{ n\}
        \end{gather*}
        It follows that, passing to the even terms of the associated graded pieces of the motivic filtration, $\grM^{n, \mathrm{even}} \TR ((S,Q);\Zp)^{hS^1}$ can be interpreted in terms of $\limr \nr^{\geq n} \Delc_{(S,Q)/ \ainf} \{ n\}$ and $\grM^{n, \mathrm{even}} \TR ((S,Q); \Zp)$ can be interpreted in terms of $\limr \nr^n \Delc_{(S,Q)/ \ainf} \{ n\}$.
    \end{enumerate}
\end{thm}

The structure of this note is as follows. In section 2, we briefly gather some prerequisites regarding log rings and log versions of topological Hochschild homology and variants. In section 3, we describe the motivic filtration of log $\TR^r$ and of its $S^1$-homotopy fixed points, in terms of the $r$-Nygaard filtration on log prismatic cohomology over a perfectoid base. Finally, in section 4, after a brief discussion of the smooth case, we apply this to the setting of \ref{thm1 Hesselholt conjectures}.

\subsection*{Acknowledgements}
The author would like to thank Prof. Akhil Mathew, for informing him about the existence of the article \cite{hesselholt2005absolute}, as well as for helpful conversations.

\section{Preliminaries}

In this section we briefly recall some facts about log rings and the log cotangent complex, the log quasisyntomic topology and how it is used to construct the motivic filtration on log $\THH$ and its variants. For this we rely on \cite{binda2023logarithmic}. In addition, we recall the definition of the $p$-complete, relative de Rham--Witt complex of Langer--Zink, in the log setting, as this was presented by \cite{matsuue2017relative}.

\subsection{Log rings}


\begin{defn}[Log rings] The following are some of the most important facts about log rings.
    \begin{enumerate}[1)]
        \item A \emph{pre-log ring} is a triple $(A, M, \alpha)$, comprising of a ring $A$, a commutative monoid $M$, and a map $\alpha : P \to R$ to the monoid associated to multiplication on $R$
        \item This is a \emph{log-ring} if $\alpha^{-1}(A^{\times}) \to A^{\times}$ is an isomorphism.
        \item A morphism between pre-log rings comprises of a map between rings and a map between monoids, so that the relevant structure commutes.
    \end{enumerate}
\end{defn}

Via the process of animation, one is also able to construct a good notion of \emph{animated (pre-) log rings}.

Since we plan to use differential forms in their various incarnations, a good notion of flatness is required.

\begin{defn}[flatness \cite{binda2023logarithmic}]
    For a map of prelog rings $(f, f^{\flat}): (A,M) \to (B,N)$, we have that:
    \begin{enumerate}
        \item $f^{\flat}$ is \emph{flat} if the map $N \oplus_M N' \to \pi_0 (N \oplus_M N')$, for all monoid maps $M \to N'$.
        \item $(f, f^{\flat})$ is \emph{homologically log flat} if both $f$ and $f^{\flat}$ are flat.
        \item $(f, f^{\flat})$ is \emph{homologically faithfully flat}, if in addition to the above conditions, $f$ is faithfully flat.
    \end{enumerate}
\end{defn}

One of the main algebraic objects that have good generalizations in the logarithmic world are de Rham differentials.

\begin{defn}[de Rham differentials]
    Let $(A,M) \to (B,N)$ be a map of pre-log rings.
    \begin{enumerate}[1)]
        \item One can define the \emph{module of relative log differentials} as the $B$-module obtained from the quotient of $\Omega_{B/A}^1 \oplus (B \otimes_{\Z} \mathrm{coker} (M^{\mathrm{gp}} \to N^{\mathrm{gp}})$ by the submodule generated by $(\mathrm{d}\beta(n),0) - (0,\beta(n)\otimes n)$, $n \in N$.
        \item The process of animation provides a good notion of log cotangent complex for a map of animated pre-log rings. This has been studied by Gabber and Olsson \cite{olsson2005logarithmic}.
        \item Using wedge powers one can have good notions of the de Rham complex and the derived de Rham complex, equipped with the Hodge filtration, in the logarithmic setting. Following \cite{BMS2}, all these aforementioned gadgets satisfy homologically log flat descent.
    \end{enumerate}
\end{defn}

In the case of $p$-complete rings, a good extension of the Hodge-filtered de Rham complex has been studied by Langer--Zink. Such an object has a good analogue in the logarithmic setting as well. These are actually true more generally in the $p$-typical setting, but for our considerations, the $p$-complete hypothesis suffices.

\begin{defn}[Log de Rham--Witt complex, after Matsuue \cite{matsuue2017relative}]
    Let $(A,M) \to (B,N)$ be a map of $p$-complete pre-log rings.
    \begin{enumerate}[1)]
        \item There exists a \emph{relative log de Rham--Witt complex} $W_r \Omega_{(B,N)/(A,M)}$, serving as an analogue of the relative de Rham--Witt complex of Langer--Zink, in the logarithmic world. This has been introduced as the universal log $\F$-$\V$-procomplex, by Matsuue.
        \item Using animation, it is possible to extend this to a relative log derived de Rham--Witt complex $\drw_{r, (B,N)/(A,M)}$, for a pair of $p$-complete animated pre-log rings $(A,M) \to (B,N)$.
    \end{enumerate}
\end{defn}

Finally, in order to be able to study the motivic filtrations on variants of $\THH$, we need to also have a good notion of quasisyntomic rings in the logarithmic world. We follow \cite{binda2023logarithmic}, again.

\begin{defn}[Log quasisyntomic \cite{binda2023logarithmic}]
    Consider an integral map of pre-log rings with bounded $p^{\infty}$-torsion. Then:
    \begin{enumerate}[1)]
        \item $(f, f^{\flat})$ is \emph{log quasismooth} (resp. \emph{log quasismooth cover}) if $A \to B$ is $p$-completely flat (resp. $p$-completely faithfully flat) and $\mathbb{L}_{(B,N) / (A,M)}$ is $p$-completely flat.
        \item $(f, f^{\flat})$ is \emph{log quasisyntomic} (resp. \emph{log quasisyntomic cover}) if $A \to B$ is $p$-completely flat (resp. $p$-completely faithfully flat) and $\mathbb{L}_{(B,N) / (A,M)}$ has $p$-complete Tor-amplitude in $[-1,0]$.
    \end{enumerate}
\end{defn}

The main examples of log quasisyntomic rings are perfectoid and log quasiregular-semiperfectoid rings.

\begin{ex}[Log quasiregular-semiperfectoid rings]
    An integral pre-log ring $(A,M)$ is called \emph{log quasiregular-semiperfectoid} if the following hold:
    \begin{enumerate}[a)]
        \item $(A,M)$ is log quasisyntomic.
        \item There exists a map $R_0 \to A$, from a perfectoid ring $R_0$.
        \item $A/p$ and $M$ are both semiperfect.
    \end{enumerate}
    Note that it is possible to also define log perfectoid rings, but for our purposes, following \cite{binda2023logarithmic}, we consider perfectoid rings with trivial log structure.
\end{ex}

The important thing about log quasiregular-semiperfectoid rings is that they are amenable for calculations, but are also the foundation for applying descent.

\begin{prop}[Log quasisyntomic descent \cite{binda2023logarithmic}]
    Log quasiregular-semiperfectoid rings, equipped with log quasisyntomic covers, provide a basis for the log quasisyntomic topology.
\end{prop}

\subsection{Variants of $\THH$ with log poles}



In order to define topological Hochschild homology with logarithmic poles, the procedure of repletion is required.

\begin{defn}
    Consider a map of pre-log rings $(A,M) \to (B,N)$, such that the induced map $M^{\mathrm{gp}} \to N^{\mathrm{gp}}$ is surjective.
    \begin{enumerate}
        \item The \emph{repletion} of $M$ is the output of the following pullback:
        \begin{equation*}
            \begin{tikzcd}[column sep=huge]
                M^{\mathrm{rep}} \arrow[r] \arrow[d] & N \arrow[d] \\
                M^{\mathrm{gp}} \arrow[r] & N^{\mathrm{gp}}
            \end{tikzcd}
        \end{equation*}
        \item The \emph{replete base change} is the construction $A^{\mathrm{rep}}:= A \otimes_{\Z [M]} \Z[M^{\mathrm{rep}}]$
    \end{enumerate}
\end{defn}

We can now define topological Hochschild homology with logarithmic poles. By a \emph{pre-log $\einfty$-ring spectrum} we mean a pair $(A,M)$, comprising of an $\einfty$-ring spectrum and a commutative monoid, together with a map $\Sph [M] \to A$, in $\Calg$.

\begin{defn}[log $\THH$, after \cite{rognes2009topological, rogneslog, binda2023logarithmic}]
    Let $(A,M)$ be a pre-log $\einfty$-ring spectrum. We define \emph{topological Hochschild homology with logarithmic poles} (or, for short, \emph{log $\THH$}) to be the following $\einfty$-ring spectrum:
    \begin{equation*}
        \THH (A,M) := \THH (A) \otimes_{\Sph [B^{\mathrm{cyc}}(M)]} \Sph [M \times_{M^{\mathrm{gp}}} B^{\mathrm{cyc}}(M^{\mathrm{gp}})]
    \end{equation*}
    Equivalently, this can be presented as:
    \begin{equation*}
        \THH (A,M) \simeq A \otimes_{A \otimes A \otimes_{\Sph [M \oplus M]} \Sph [(M \oplus M)^{\mathrm{rep}}] } A
    \end{equation*}
\end{defn}

\begin{prop}[Cyclotomic structure]
    Let $(A,M)$ be a pre-log $\einfty$-ring spectrum. Then $\THH (A,M)$ is a cyclotomic spectrum:
    \begin{enumerate}[a)]
        \item It admits a canonical $S^1$-action, via the simplicial presentation of the circle: $S^1 \simeq * \amalg_{* \amalg *} *$
        \item For all primes $p$, there exist $S^1$-equivariant higher Frobenii maps, using the repletion induced Frobenius:
        \begin{equation*}
            \begin{tikzcd}
                \THH (A) \otimes_{\Sph [B^{\mathrm{cyc}} (M)] } \Sph [B^{\mathrm{rep}} (M)] \arrow[r] \arrow[dr, "\varphi_p"'] & \THH (A)^{tC_p} \otimes_{\Sph [B^{\mathrm{cyc}} (M)]^{tC_p}} \Sph[B^{\mathrm{rep}}(M)]^{tC_p} \arrow[d] \\
                & \THH (A,M)^{tC_p}
            \end{tikzcd}
        \end{equation*}
    \end{enumerate}
    Also, note that if $A$ is a connective $\einfty$-ring spectrum, then $\THH (A,M)$ is a connective $\einfty$-ring spectrum as well.
\end{prop}

Replacing the sphere spectrum with the integers and the world of $\einfty$-ring spectra with that of (animated) rings gives rise to plain Hochschild homology.

\begin{defn}
    Let $(A,M)$ be an (animated) pre-log ring. Then its \emph{Hochschild homology with logarithmic poles} (or, for short, \emph{log $\HH$}) is the following (animated) ring, equipped with an $S^1$-action:
    \begin{align*}
        \HH (A,M) := A \otimes_{(A \otimes A)^{\mathrm{rep}}} A \\
        \simeq \THH (A,M) \otimes_{\THH (\Z)} \Z
    \end{align*}
\end{defn}


\begin{prop}
    Given the fact that for a pre-log $\einfty$-ring spectrum (A,M), $\THH (A,M)$ is a cyclotomic spectrum, we can define as usual:
    \begin{equation*}
        \begin{tikzcd}[row sep=tiny]
            \TC^- (A,M) := \THH (A,M)^{hS^1} \quad \TP (A,M) := \THH (A,M)^{tS^1} \\
            \TC ((A,M);\Zp) := \mathrm{map}_{\mathrm{CycSp}} (\Sph, \THH ((A,M);\Zp)) \simeq \fib \Big( \can - \varphi_p^{hS^1} : \TC^- ((A,M);\Zp) \longrightarrow \TP((A,M);\Zp) \Big)
        \end{tikzcd}
    \end{equation*}
    The spectra $\THH (-,-)$, $\THH (-,-)_{hC_n}$, $\THH (-,-)^{hC_n}$, $\THH (-,-)^{tC_n}$, $\THH (-,-)_{hS^1}$, $\TC^- (-,-)$, $\TP (-,-)$, $\TC (-,-)$ all satisfy descent for the homologically log flat topology.
\end{prop}

Now it is possible to reproduce the main result of \cite{BMS2}, in the logarithmic setting, due to \cite{binda2023logarithmic}.

\begin{thm}[Motivic filtrations in the log setting, after \cite{binda2023logarithmic}]
    Let $R_0$ be a perfectoid ring and $(S,Q)$ be a $p$-complete log quasisyntomic $R_0$-algebra. Then the invariants of the following square are equipped with $\Z$-indexed, complete, multiplicative, motivic filtrations $\FilM^{\bullet} (-)$:
    \begin{equation*}
        \begin{tikzcd}[column sep=huge]
            \TC^- ((S,Q);\Zp) \arrow[r, "\varphi^{hS^1}"] \arrow[d] & \TP ((S,Q);\Zp) \arrow[d] \\
            \THH ((S,Q);\Zp) \arrow[r, "\varphi"] & \THH ((S,Q);\Zp)^{tC_p}
        \end{tikzcd}
    \end{equation*}
    Passing to the $i$-th associated graded piece, one obtains a commutative diagram, whose rows represent the divided Frobenius and its graded counterpart in the setting of log prismatic cohomology:
    \begin{equation*}
        \begin{tikzcd}[column sep=huge]
            \n^{\geq n} \Delc_{(S,Q)/ \ainf} \{ n\} [2n] \arrow[r, "\varphi_n"] \arrow[d] & \Delc_{(S,Q)/ \ainf} \{ n\} [2n] \arrow[d] \\
            \n^n \Delc_{(S,Q)/ \ainf} \{ n\} [2n] \arrow[r, "\gr \varphi_n"] & \overline{\Del}_{(S,Q)/ \ainf} \{ n\} [2n]
        \end{tikzcd}
    \end{equation*}

    These motivic filtrations induced spectral sequences, which converge on the second page, and thus calculate the homotopy groups of the aforementioned invariants.
    
    Note that in the case of $(S,Q)$ being log quasiregular-semiperfectoid, all these invariants are concentrated on even homotopy groups (which identify exactly with prismatic invariants) and, hence, the motivic filtrations are identified with the double-speed Postnikov ones. In this setting, the spectral sequences of the invariants of the top row are the $S^1$-homotopy fixed points/Tate spectral sequences, respectively.
\end{thm}

\section{TR with logarithmic poles and its motivic filtration}

In this section we introduce the $p$-typical $\TR^r$ with logarithmic poles and record its main properties. Then, inspired by \cite{Andri1}, we study its motivic filtration by passing to its $S^1$-homotopy fixed points.

\subsection{TR with logarithmic poles}



\begin{defn}
    Let $(A,M)$ be a pre-log $\einfty$-ring spectrum. For $1 \leq r \leq \infty$, define its \emph{$p$-typical topological restriction homology} to be the invariant obtained via the following iterated product construction:
    \begin{equation*}
        \TR^r (A,M) := \THH (A,M)^{hC_{p^{r-1}}} \times_{(\THH (A,M)^{tC_p})^{hC_{p^{r-2}}}} \dots \times_{\THH (A,M)^{tC_p}} \THH (A,M)
    \end{equation*}
    where the maps on the left are the canonical ones, while the maps on the right correspond to higher Frobenii $\varphi^{hC_{p^i}}$.

    This is equipped with Restriction, Frobenius, and Verschiebung operators, as in the non logarithmic case. Taking the limit with respect to the Restriction maps, we obtain:
    \begin{gather*}
        \TR (A,M) := \limr \TR^r (A,M)  \\
        \simeq \dots \times_{(\THH (A,M)^{tC_p})^{hC_{p^{r-1}}}} \THH (A,M)^{hC_{p^{r-1}}} \times_{(\THH (A,M)^{tC_p})^{hC_{p^{r-2}}}} \dots \times_{\THH (A,M)^{tC_p}} \THH (A,M)
    \end{gather*}
    Note that if $A$ is a connective $\einfty$-ring spectrum, then the same is also true for $\TR^r (A,M)$, $1 \leq r \leq \infty$.
\end{defn}

From now one we restrict to using $\Zp$-coefficients everywhere. Needless to say that most of the constructions that follow can be done without this restriction, but over $\Zp$ there is a substantial simplification as a result of the Tate vanishing lemma \cite{NS}.

\begin{con}
    Let $(A,M)$ be a pre-log $\einfty$-ring spectrum. Then, for $1 \leq r \leq \infty$, the iterated product construction over $\Zp$-coefficients simplifies to the following:
    \begin{equation*}
        \TR^r ((A,M);\Zp) \simeq \THH ((A,M);\Zp)^{hC_{p^{r-1}}} \times_{\THH ((A,M);\Zp)^{tC_{p^{r-1}}}} \dots \times_{\THH ((A,M);\Zp)^{tC_p}} \THH ((A,M);\Zp)
    \end{equation*}
    Applying $S^1$-homotopy fixed points, we have the following iterated product description of $TR^r ((A,M);\Zp)^{hS^1}$:
    \begin{equation*}
        \TR^r ((A,M);\Zp)^{hS^1} \simeq \TC^- ((A,M);\Zp) \times_{\TP ((A,M);\Zp)} \dots \times_{\TP ((A,M);\Zp)} \TC^- ((A,M);\Zp)        
    \end{equation*}
    This fits into the following commutative diagram:
    \begin{equation} \label{TR^r S^1-homotopy fixed points main diagram}
        \begin{tikzcd}
            \TR^r ((A,M);\Zp)^{hS^1} \arrow[r] \arrow[d] & \TC^- ((A,M);\Zp) \arrow[r, "\varphi^{hS^1}"] \arrow[d] &[20pt] \TP ((A,M);\Zp) \arrow[d] \\
            \TR^r ((A,M);\Zp) \arrow[r] & \THH ((A,M);\Zp)^{hC_{p^{r-1}}} \arrow[r, "\varphi^{hC_{p^{r-1}}}"] & \THH ((A,M);\Zp)^{tC_{p^r}}
        \end{tikzcd}
    \end{equation}

    We also have natural Restriction, Frobenius, and Verschiebung maps induced on the $S^1$-homotopy fixed points:
    \begin{equation*}
        \begin{tikzcd}
            \TR^{r+1} ((A,M);\Zp)^{hS^1} \arrow[r, "\Res^{hS^1}", "\F^{hS^1}"'] \arrow[d] & \TR^r ((A,M);\Zp)^{hS^1} \arrow[d] &[-20pt] \TR^r ((A,M);\Zp)^{hS^1} \arrow[r, "\V^{hS^1}"] \arrow[d] & \TR^{r+1} ((A,M);\Zp)^{hS^1} \arrow[d] \\
            \TR^{r+1} ((A,M);\Zp) \arrow[r, "\Res", "\F"'] & \TR^r ((A,M);\Zp) & \TR^r ((A,M);\Zp) \arrow[r, "\V"] & \TR^{r+1} ((A,M);\Zp)
        \end{tikzcd}
    \end{equation*}
\end{con}

\subsection{The motivic filtration}

Since we are working over a perfectoid base, let us recall the motivic filtration of $\TR^r$ of perfectoids.

\begin{prop}[$\TR^r$ of perfectoids, \cite{BMS2, K1localTR, Andri1}]
    Let $R_0$ be a perfectoid ring. For $1 \leq r < \infty$, the invariants of the following commutative diagram are concentrated degrees:
    \begin{equation} \label{TR^r diagram for perfectoid}
        \begin{tikzcd}
            \TR^r (R_0;\Zp)^{hS^1} \arrow[r] \arrow[d] & \TC^- (R_0;\Zp) \arrow[r, "\varphi^{hS^1}"] \arrow[d] &[20pt] \TP (R_0;\Zp) \arrow[d] \\
            \TR^r (R_0;\Zp) \arrow[r] & \THH (R_0;\Zp)^{hC_{p^{r-1}}} \arrow[r, "\varphi^{hC_{p^{r-1}}}"] & \THH (R_0;\Zp)^{tC_{p^r}}
        \end{tikzcd}
    \end{equation}
    Passing to homotopy groups, we obtain the following diagram:
    \begin{equation*}
        \begin{tikzcd}[column sep=huge, row sep=large]
            \ainf (R_0) [u_r, v_r] \arrow[r] \arrow[d, "\vartheta_r \text{-linear}", "v_r \mapsto 0"'] & \ainf (R_0) [\sigma, \sigma^{-1}] \arrow[d,"\widetilde{\vartheta}_r \text{-linear}"] \\
            \W_r (R_0)[u_r] \arrow[r] & \W_r (R_0)[\sigma, \sigma^{-1}]
        \end{tikzcd}
    \end{equation*}
    where the elements $u_r, \sigma$ live in degree $2$ and $v_r$ lives in degree $-2$.

    In this situation, the motivic filtrations of all the invariants are the double speed Postnikov ones. Regarding $\TR^r$, the $S^1$-homotopy fixed point spectral sequence degenerates and equips $\pi_0 \TR^r (R_0;\Zp)^{hS^1} \simeq \ainf (R_0)$ with the $r$-Nygaard filtration. This is no other than the $\xi_r$-adic filtration. It follows that the $n$-th associated graded terms of the motivic filtration, which are the $n$-th homotopy groups, give rise to the following commutative diagram:
    \begin{equation*}
        \begin{tikzcd}
            \xi_r^n \ainf (R_0) \{ n\} \arrow[r, "\varphi_{r, n}"] \arrow[d] & \ainf (R_0) \{ n\} \arrow[d] \\
            \xi_r^n/ \xi_r^{n+1} \ainf (R_0) \{ n\} \arrow[r] & \W_r (R_0) \{ n\}
        \end{tikzcd}
    \end{equation*}
    
    For the case $r=\infty$, the invariants $\TR (R_0;\Zp)$ and $\TR (R_0;\Zp)^{hS^1}$ are not concentrated on even degrees, in general. However, by taking the limits over Restriction maps, from the case $r<\infty$, we are able to see that these invariants are equipped with motivic filtrations, which are no other than the double-speed Postnikov ones. Passing to the associated graded pieces, we obtain the following identification:
    \begin{equation*}
        \begin{cases}
            \grM^n \TR (R_0;\Zp)^{hS^1} \simeq \rlimr \xi_r^n \ainf (R_0) \{ n\} \\[5pt]
            \grM^n \TR (R_0;\Zp) \simeq \rlimr \xi_r^n / \xi_r^{n+1} \ainf (R_0) \{ n\}
        \end{cases}
    \end{equation*}
    In particular, the even homotopy groups correspond to the $\lim$ terms, while the $\lim^1$ terms are identified with the odd homotopy groups. One particular instance for which these latter vanish, and hence the invariants are even, occurs when $R_0 = \mathcal{O}_C$, for a spherically complete $C$.
\end{prop}

This result essentially uses the fact that we can go back and forth between $\TR^r (R_0;\Zp)$ and its $S^1$-homotopy fixed points. This is also true in the logarithmic world:
\begin{prop} \label{Going back to TR^r}
    Let $R_0$ be a perfectoid ring and $(A,M)$ be a pre-log $\einfty$-ring spectrum, with $A$ connective. Then, for $1 \leq r < \infty$, we have that:
    \begin{equation*}
        \begin{tikzcd}
            \TR^r ((A,M);\Zp)^{hS^1} /v_r \simeq \TR^r ((A,M);\Zp)^{hS^1} \otimes_{\TR^r (R_0;\Zp)^{hS^1}} \TR^r (R_0;\Zp) \arrow[r, "\simeq"] & \TR^r ((A,M);\Zp)
        \end{tikzcd}
    \end{equation*}
\end{prop}

Now we can move on to the motivic filtrations for the case of any $p$-complete pre-log quasisyntomic ring.
\begin{proof}[Proof of Theorem \ref{thm2 motivic filtrations}]
    Let $(S,Q)$ be a pre-log quasiregular-semiperfectoid ring and let $2 \leq r < \infty$ (the case $r=1$ is already settled). Remember that in this case, the following invariants are even:
    \begin{equation*}
        \TC^- ((S,Q);\Zp), \quad \TP ((S,Q);\Zp), \quad \THH ((S,Q);\Zp), \quad \THH ((S,Q);\Zp)^{tC_{p^r}}, \; 1 \leq r < \infty
    \end{equation*}

    From the iterated pullback formula of $\TR^r ((S,Q);\Zp)^{hS^1}$ (resp. of $\TR^r ((S,Q);\Zp)$) it follows that it is equipped with a motivic filtration, which is no other than the double-speed Postnikov one. Passing to the associated graded pieces, we get a two-term complex. Its cohomology groups, which are the even/odd homotopy groups of $\TR^r ((S,Q);\Zp)^{hS^1}$ (resp. of $\TR^r ((S,Q);\Zp)$) fit in the following exact sequence of abelian groups:
    \begin{equation*}
        \begin{tikzcd}[row sep=tiny]
            0 \arrow[r] & \pi_{2n} \TR^r ((S,Q);\Zp)^{hS^1} \arrow[r] & \pi_{2n} \TC^- ((S,Q);\Zp) \times \pi_{2n} \TR^{r-1} ((S,Q);\Zp)^{hS^1} \arrow[r] & \, \\
            \, \arrow[r] & \pi_{2n} \TP ((S,Q);\Zp) \arrow[r] & \pi_{2n-1} \TC^- ((S,Q);\Zp) \times \pi_{2n-1} \TR^{r-1} ((S,Q);\Zp)^{hS^1} \arrow[r] & 0
        \end{tikzcd}
    \end{equation*}

    Recursively, we see that the even homotopy groups give rise to a filtration, that we call the \emph{$r$-Nygaard filtration on log prismatic cohomology}, which has the following iterated pullback description:
    \begin{gather*}
        \nr^{\geq n} \Delc_{(S,Q)/ \ainf} \{ n\} \simeq \\
        \n^{\geq n} \Delc_{(S,Q)/ \ainf} \{ n\} \times_{\Delc_{(S,Q)/ \ainf} \{ n\}} \dots \times_{\Delc_{(S,Q)/ \ainf} \{ n\}} \n^{\geq n} \Del_{(S,Q)/ \ainf} \{ n\}
    \end{gather*}

    Going back to $\TR^r ((S,Q);\Zp)$ via Proposition \ref{Going back to TR^r}, it follows that its even homotopy groups can be identified with the associated graded terms for the $r$-Nygaard filtration:
    \begin{equation*}
        \pi_{2n} \TR^r ((S,Q);\Zp) \simeq \nr^{\geq n} \Delc_{(S,Q)/ \ainf} \{ n\}
    \end{equation*}

    In particular, we have the following commutative diagram:
    \begin{equation*}
        \begin{tikzcd}
            \TR^r ((S,Q);\Zp)^{hS^1} \arrow[r] \arrow[d] & \TC^- ((S,Q);\Zp) \arrow[r, "\varphi^{hS^1}"] \arrow[d] &[20pt] \TP((S,Q);\Zp) \arrow[d] \\
            \TR^r ((S,Q);\Zp) \arrow[r] & \THH ((S,Q);\Zp)^{hC_{p^{r-1}}} \arrow[r, "\varphi^{hC_{p^{r-1}}}"] & \THH ((S,Q);\Zp)^{tC_{p^r}}
        \end{tikzcd}
    \end{equation*}
    This gives rise to the following commutative square, when passing to the even homotopy groups of the invariants:
    \begin{equation*}
        \begin{tikzcd}[column sep=huge]
            \nr^{\geq n} \Delc_{(S,Q)/ \ainf} \{ n\} \arrow[r, "\varphi_{r, n}"] \arrow[d] & \Delc_{(S,Q)/ \ainf} \{ n\} \arrow[d] \\
            \nr^n \Delc_{(S,Q)/ \ainf} \{ n\} \arrow[r] & \Del_{(S,Q)/ \ainf}^{\HT, r} \{ n\}
        \end{tikzcd}
    \end{equation*}
    The upper row corresponds to a divided version of the $r$-th iteration of Frobenius, while the lower row is a divided version of this.

    While the even homotopy groups (or if you wish, the even parts of the motivic filtration) behave well, the odd homotopy groups are junk terms. In \cite{Andri1}, we were able to show that in the non logarithmic setting, they vanish locally for the quasisyntomic topology, after applying the odd vanishing of \cite{BS}. We do not know how to do this in the log setting, therefore, we simply ignore the terms.
\end{proof}

\section{Hesselholt's conjectures over $\mathcal{O}_C$}







In this section, we focus on understanding the relation between log $\TR^r$ and the $p$-complete, relative log de Rham--Witt complex. Note that in what follows, we freely use the fact that the algebraic approach and the homotopy theoretic approach to log prismatic cohomology coincide. Our results are conditional on that.

\subsection{The log de Rham--Witt comparison}

In this section we restrict to working over the perfectoid base $\mathcal{O}_C$, where $C$ is an algebraically closed complete extension of $\Q_p$. Our goal is to briefly explain how the $p$-complete, relative, $r$-truncated log de Rham--Witt complex of Matsuue \cite{matsuue2017relative} relates with the $r$-Nygaard filtration. We build on the ideas of \cite{Andri1}, from the non logarithmic world, and on \cite{vcesnavivcius2019cohomology, aoki2023p}, from the logarithmic setting.

\begin{con}[Log $r$-Nygaard filtration] \label{Log $r$-Nygaard filtration}
    Let $(S,Q)$ be an integral log smooth $p$-complete log ring over $\mathcal{O}_C$. In this particular case, its log prismatic cohomology $\Delc_{(S,Q)/\ainf}$ gets identified with the log $A\Omega$ cohomology of \cite{vcesnavivcius2019cohomology}. It follows that the Nygaard filtration on the later is induced by the d\'ecalage functor with respect to the element $\xi$. Therefore, and in analogy with the results of \cite{Andri1}, the iterated pullback construction for $\nr^{\geq n} \Delc_{(S,Q)/ \ainf}$ gets identified with the d\'ecalage induced filtration $L\eta_{\xi_r}^{\geq n} A\Omega_{(S,Q)}$.
\end{con}

\begin{prop}[Relative log de Rham--Witt comparison] \label{Log de Rham--Witt comparison} Assume that we are in the setting of \cite{vcesnavivcius2019cohomology}, that is we work with a semistable $p$-adic formal scheme $\spf S$. Then as shown in \cite[Ex. 1.5]{koshikawa2020logarithmic}, its logarithmic prismatic cohomology gets identified with log $A\Omega$-cohomology. Following the previous construction \ref{Log $r$-Nygaard filtration}, the $r$-divided Frobenius lives in the following commutative square:
\begin{equation*}
    \begin{tikzcd}
        \nr^{\geq n} \Del_{(S,Q)/ \ainf} \{ n\} \arrow[rr, "\varphi_{r, n}"] \arrow[d] & & \Del_{(S,Q)/ \ainf} \{ n\} \arrow[d] \\
        \nr^n \Del_{(S,Q)/ \ainf} \{ n\} \arrow[rr] \arrow[dr, bend right=10, "\simeq"] & & \Del_{(S,Q)/ \ainf}^{\HT, r} \{ n\} \\[-20pt]
        & \Fil_n^{\conj} \Del_{(S,Q)/ \ainf}^{\HT, r} \{ n\} \arrow[ur, hook, bend right=10]
    \end{tikzcd}
\end{equation*}
The lower row factors through the conjugate filtration on $r$-Hodge--Tate cohomology, which in this case comes from the Postnikov tower:
\begin{equation*}
    \Fil_n^{\conj} \Del_{(S,Q)/ \ainf}^{\HT, r} := \tau^{\leq n} \Del_{(S,Q)/ \ainf} \otimes_{\ainf}^{\mathbb{L}} \ainf / \widetilde{\xi}_r
\end{equation*}
Passing to graded pieces, we have the following description in terms of the $p$-complete, relative log de Rham--Witt forms of Matsuue:
\begin{equation*}
    \gr_n^{\conj} \Del_{(S,Q)/ \ainf}^{\HT, r} \simeq W_r\Omega_{(S,Q)/(\mathcal{O}_C, Q_{\mathcal{O}_C})}^{n, \mathrm{cont}}
\end{equation*}
\end{prop}

\begin{proof}[Proof sketch]
    The idea is contained in \cite[Sec. 11]{BMS1} in the non-logarithmic case and in \cite[Sec. 4]{aoki2023p} in the logarithmic case. In particular, taking quotients with the elements $\xi_r$ and using the generalities of the d\'ecalage functor and the Beilinson $t$-structure, one is able to construct a $p$-complete log $\F$-$\V$-procomplex. The identification with the $p$-complete, relative de Rham--Witt forms reduces to a local computation, as in loc. cit.
    
    The Bockstein differential gives rise to the entire complex, which identifies with the $r$-truncated, $p$-complete, log relative de Rham--Witt complex, again via generalities on the d\'ecalage functor. It follows that the $n$-th associated graded of the $r$-Nygaard filtration identify with the $n$-th filtered term of the log $r$-Hodge--Tate cohomology, thus making the diagram above commutative.

    Note that this argument should generalize to the settings of log $A\Omega$-cohomology of \cite{diao2024logarithmic} and/or that of log prismatic cohomology \cite{koshikawa2020logarithmic, koshikawa2023logarithmic}, when working over a perfect prism, in analogy with the non logarithmic case of \cite{molokov2020prismatic}.
\end{proof}

\subsection{$\TR^r$ of smooth rings over $\mathcal{O}_C$}

In this section we restrict our focus on phenomena related to Theorem \ref{thm1 Hesselholt conjectures}.

\begin{con}[Hesselholt--Madsen fibre sequences \cite{hesselholt2003k, hesselholt2004rham}]
    Let $X$, $U$, and $Y$ be as in the statement of Theorem \ref{thm1 Hesselholt conjectures}. In their work \cite{hesselholt2003k, hesselholt2004rham}, Hesselholt--Madsen construct a fibre sequence for topological restriction homology, involving log $\TR^r$ on the cofibre term. Moreover, this receives a trace map from the localization sequence of algebraic $\K$-theory, as follows:
    \begin{equation*}
        \begin{tikzcd}[column sep=huge]
            \K (Y) \arrow[r, "i^!"] \arrow[d] & \K (X) \arrow[r, "j_*"] \arrow[d] & K (U) \arrow[d] \\
            \TR^r (Y) \arrow[r, "i^!"] & \TR^r (X) \arrow[r, "j_*"] & \TR^r ((X,U))
        \end{tikzcd}
    \end{equation*}
    Note that a priori, the definition of logarithmic structures in $\TR^r$ by Hesselholt--Madsen is different from the one that we follow here. The discussion on log $\THH$ presented in \cite{hesselholt2019arbeitsgemeinschaft} fixes this issue.

    Passing to the limit with respect to the Restriction maps and taking the fixed points of Frobenius to form topological cyclic homology, yields the following:
    \begin{equation*}
        \begin{tikzcd}[column sep=huge]
            \K (Y) \arrow[r, "i^!"] \arrow[d] & \K (X) \arrow[r, "j_*"] \arrow[d] & K (U) \arrow[d] \\
            \TC (Y) \arrow[r, "i^!"] & \TC (X) \arrow[r, "j_*"] & \TC ((X,U))
        \end{tikzcd}
    \end{equation*}
\end{con}

Working with $\Zp$ coefficients and with $p$
-complete objects, we have that the motivic filtration of algebraic $\K$-theory is identified, locally with respect to the quasisyntomic topology, identified with the double-speed Postnikov one \cite[Cor. 14.2]{BS}, \cite[Ex. 1.6]{BM}. Therefore, in the first two columns in the above diagram above, this filtration on algebraic $\K$-theory is \'etale locally identified with that coming from $\TC$, under the trace map. Hence, the same happens for the third column, as well.

Arguing as in \cite[Thm. 1.5]{Andri1}, we have the following corollary:

\begin{cor}[Odd vanishing] \label{Odd vanishing}
    The odd parts of the associated graded pieces for the motivic filtration $\grM^{\bullet, \mathrm{odd}} \TR^r ((X,U);\Zp)$ vanish locally for the \'etale topology.
\end{cor}

Now we can apply these in the setting of Hesselholt's conjectures over $\mathcal{O}_C$.

\begin{proof}[Proof sketch for Theorem \ref{thm1 Hesselholt conjectures}]
    The motivic filtration and spectral sequence claims follow by applying Theorem \ref{thm2 motivic filtrations} together with Corollary \ref{Odd vanishing}. The claim regarding the $p$-complete, relative log de Rham--Witt forms follow by application of Proposition \ref{Log de Rham--Witt comparison}.
    
    Finally, regarding the fibre sequence, its existence follows by taking the fibre of the equalizer of the identity and Frobenius on log $\TR$, yielding log $\TC$. After \cite{BMS2}, syntomic cohomology comes as $p$-adic \'etale motivic cohomology and, hence, the left term is identified as such from the localization sequence of algebraic $\K$-theory mapping to the log $\TC$ fire sequence (see the relevant discussion after \cite[Theorem 1.5.6]{hesselholt2003k}). The identification of syntomic cohomology with the \'etale Tate twists comes, for example, from the main results of \cite{BM, bouis2023cartier}.
\end{proof}

\printbibliography
\end{document}